\documentclass[reqno,11pt]{amsart}

\usepackage{color}

\newtheorem{theorem}{Theorem}[section]
\newtheorem{lemma}[theorem]{Lemma}
\newtheorem{proposition}[theorem]{Proposition}

\theoremstyle{definition}

\theoremstyle{remark}
\newtheorem{remark}[theorem]{Remark}

\newcommand{\bR}{\mathbb R}

\newcommand\cL{\mathcal{L}}

\def\bH{\mathbb{H}}
\def\cH{\mathcal{H}}

\newcommand{\esssup}{\operatorname*{ess\,sup}}
\renewcommand{\epsilon}{\varepsilon}

\begin{document}
\title[Partial regularity for NSE]{Partial Regularity of solutions to the High dimensional Navier-Stokes Equations near boundary}
\author[H. Dong]{Hongjie Dong}
\address[H. Dong]{Division of Applied Mathematics, Brown University,
182 George Street, Providence, RI 02912, USA}
\email{Hongjie\_Dong@brown.edu}

\author[X. Gu]{Xumin Gu}
\address[X. Gu]{School of Mathematical Sciences, Fudan University,
Shanghai 200433, People's Republic of China}
\email{xumin\_gu@brown.edu;11110180030@fudan.edu.cn}

\begin{abstract}
We consider suitable weak solutions of the incompressible Navier--Stokes equations in two cases: the 4D time-dependent case and the 6D stationary case. We prove that up to the boundary, the two-dimensional Hausdorff measure of the set of singular points is equal to zero in both cases.
\end{abstract}



\maketitle

\section{Introduction}
In this paper we consider two cases of the high dimensional incompressible Navier--Stokes equations with unit viscosity and an external force:
the 4D time-dependent case
\begin{equation}
\label{ns}
\left\{\begin{aligned}
u_t +u\cdot \nabla u -\Delta u + \nabla p &= f,\\
\nabla \cdot u &= 0
\end{aligned}\right.
\end{equation}
in a cylindrical domain $Q_T:=\Omega\times(0,T)$, where $\Omega \subset \mathbb{R}^4$ is smooth, and the 6D stationary case
\begin{equation}
\label{ns6d}
\left\{\begin{aligned}
u\cdot \nabla u -\Delta u + \nabla p &= f,\\
\nabla \cdot u &= 0
\end{aligned}\right.
\end{equation}
in a smooth domain $\Omega \subset \bR^6$.

In both cases, we assume that $u$ satisfies the zero Dirichlet boundary condition:
\begin{equation}
\label{bd}
u=0,\quad \forall\,\,x\in \partial\Omega.
\end{equation}
We are interested in the partial regularity of suitable weak solutions $(u,p)$ to \eqref{ns} or \eqref{ns6d} up to the boundary.

We say that a pair of functions $(u,p)$ is a suitable weak solution to \eqref{ns} in $Q_T$ with the boundary condition \eqref{bd} if $u \in L_{\infty}(0,T;L_2(\Omega))\cap L_{2}(0,T;W^1_2(\Omega))$ and $p \in L_{3/2}(Q_T)$ satisfy \eqref{ns} in the weak sense and additionally the generalized local energy inequality holds for any non-negative functions $\psi \in C^{\infty}(\bar\Omega\times(0,T])$ and $t\in (0,T]$:
\begin{multline}
\esssup_{0<s\leq t}\int_{\Omega}|u(x,s)|^2\psi(x,s) \,dx +2\int_{Q_t}|\nabla u|^2\psi \,dx\,ds\\
\leq \int_{Q_t}|u|^2(\psi_t+\Delta \psi)+(|u|^2+2p)u\cdot \nabla \psi +2f\cdot u \psi \,dx\, ds.
\label{energy}
\end{multline}
Similarly, we call $(u,p)$ a suitable weak solution to \eqref{ns6d} with the boundary condition \eqref{bd} if $u \in H^1(\Omega)$ and $p \in L_{3/2}(\Omega)$  satisfy \eqref{ns6d} in the weak sense and the generalized local energy inequality holds for any non-negative functions $\psi \in C^{\infty}(\bar \Omega)$:
\begin{equation*}
2\int_{\Omega}|\nabla u|^2\psi \,dx
\leq \int_{\Omega}(|u|^2\Delta \psi+(|u|^2+2p)u\cdot \nabla \psi +2f\cdot u \psi) \,dx.
\end{equation*}

One of our main results is that, for any suitable weak solution $(u,p)$ to \eqref{ns} with the boundary condition \eqref{bd}, the two dimensional space-time Hausdorff measure of the set of singular points up to the boundary is equal to zero. We also prove a similar result for the 6D stationary case.

The problem of the global regularity of solutions to the Navier--Stokes equations in three or higher space dimensions is a fundamental question in the theory of fluid equations and is still widely open. Meanwhile, many authors have studied the partial regularity of solutions. In the three dimensional case, Scheffer established various regularity results for weak solutions in \cite{Scheffer_76, Scheffer_77}. In a celebrated paper \cite{CKN_82}, Caffarelli, Kohn, and Nirenberg firstly introduced the notion of suitable weak solutions, which satisfy a generalized local energy inequality and $p\in L_{5/4}$. They proved that for any suitable weak solution, there is an open subset where the velocity field $u$ is regular and the 1D Hausdorff measure of the complement of this subset is equal to zero. In \cite{Lin_98}, Lin gave a more direct and simplified proof of Caffarelli, Kohn, and Nirenberg's result, assuming that $p\in L_{3/2}$ and the  external force is zero. Ladyzhenskaya and Seregin \cite{Lady_99} gave another short but detailed proof of the results in \cite{CKN_82, Lin_98}. We also refer readers to Tian and Xin \cite{Tian_99}, Katz and Pavlovi\'c
\cite{KP02}, Seregin \cite{Seregin_06}, Gustafson, Kang, and Tsai \cite{Gus_07}, Vasseur \cite{Vasseur_07}, Kukavica \cite{Kukavica_08}, and the references therein for extended results. The key step in the proofs of partial regularity results is to establish certain $\varepsilon$-regularity criteria, which have many other important applications. For instance, they were crucially used by Escauriaza, Seregin, and \v Sver\'ak \cite{ESS} to prove the regularity of $L_{3,\infty}$-solutions to the 3D Navier--Stokes equations, and recently by Jia and \v Sver\'ak \cite{JS12} to construct forward self-similar solutions from arbitrary $(-1)$-homogeneous initial data.

For the four or higher dimensional Navier--Stokes equations, the problem is more super-critical. We refer readers to Scheffer \cite{Scheffer_78}, Dong and Du \cite{Dong_07} for some results for the 4D time-dependent case. For the 5D stationary case which is dimensionally analogous to the 3D time-dependent problem, Struwe \cite{struwe_88} proved that suitable weak solutions are regular outside a singular set of zero 1D Hausdorff measure. The existence of regular solutions to the stationary Navier-Stokes in high dimensions has been studied by several authors. We only refer the reader to \cite{Str95, FR96,FS09} and the references therein.

Recently, Dong and Strain \cite{Dong_12} studied the interior partial regularity for suitable weak solutions of the 6D stationary Navier--Stokes equations and proved that solutions are regular outside a singular set of zero 2D Hausdorff measure. This result was extended to the 4D time-dependent Navier--Stokes equations in \cite{Dong_13}. The main idea in \cite{Dong_12, Dong_13} is to first establish a weak decay estimate of certain scale-invariant quantities, and then successively improve this decay estimate by a bootstrap argument and the elliptic or parabolic regularity theory. The proofs therefore do not involve any compactness argument. The compactness arguments in the blowup procedure used, for instance, in  \cite{Lin_98, Lady_99} break down in the higher dimensional cases.

Partial regularity up to the boundary was studied by Scheffer \cite{Sch82} for the 3D time-dependent Navier--Stokes equations, who proved that at each time slice, $u$ is locally bounded up to the boundary except for a closed set whose 1D Hausdorff measure is finite.
Caffarelli, Kohn, and Nirenberg's partial regularity result for the 3D time-dependent Navier--Stokes equations was extended up to the flat boundary by Seregin \cite{Seregin_02} and to the $C^2$ boundary by Seregin, Shilkin, and  Solonnikov \cite{Seregin_04}.  In the 5D stationary case, Kang \cite{kang_06} proved the partial regularity up to the boundary, which extended the result by Struwe \cite{struwe_88}. See also recent Wolf \cite{Wolf_11}, Mikhailov \cite{Mikhailov_11}, and the references therein. Motivated by these paper, the objective of the current paper is to extend the aforementioned interior estimates in \cite{Dong_12, Dong_13} to the boundary case.

Next we state our main results for the 4D time-dependent case and the 6D stationary case. For the 4D time-dependent case, we have the following two boundary $\varepsilon$-regularity criteria when the boundary is locally flat. The notation in Theorems \ref{mainthm} and \ref{th2} are introduced in Section \ref{s1}.
\begin{theorem}
\label{mainthm}
Let $\Omega$ be a domain in $\mathbb{R}^4$ and $f \in L_{6}(Q_T)$.  Let $(u,p)$ be a suitable weak solution of \eqref{ns} in $Q_T$ with the boundary condition \eqref{bd}. There is a positive number $\epsilon_0$ satisfying the following property. Assume that for a point ${\hat z}=({\hat x},{\hat t})$, where ${\hat x}\in \partial\Omega$, we have $\omega({\hat z},R)=Q^{+}({\hat z},R)$ for some small $R$ and the inequality
\begin{equation*}
\limsup_{r\searrow 0} E(r) \leq \epsilon_0
\end{equation*}
holds. Then $u$ is H\"older continuous near ${\hat z}$.
\end{theorem}

\begin{theorem}
\label{th2}
Let $\Omega$ be a domain in $\mathbb{R}^4$ and $f \in L_{6}(Q_T)$. Let $(u,p)$ be a suitable weak solution of \eqref{ns} in $Q_T$ with the boundary condition \eqref{bd}.
There is a positive number $\epsilon_0$ satisfying the following property. Assume that for a point ${\hat z}=({\hat x},{\hat t})$, where ${\hat x}\in \partial\Omega$, we have $\omega({\hat z},R)=Q^{+}({\hat z},R)$ for some small $R$, and for some $\rho_0 >0$ we have
\begin{equation*}
C(\rho_0)+D(\rho_0)+G(\rho_0) \leq \epsilon_0.
\end{equation*}
Then $u$ is H\"older continuous near ${\hat z}$.
\end{theorem}
We note that according to our definitions of $\omega({\hat z},R)$ and $Q^{+}({\hat z},R)$ in Section \ref{s1}, the condition that $\omega({\hat z},R)=Q^{+}({\hat z},R)$ for some small $R$ implies that the boundary is locally flat near the point ${\hat x}$.

Our next result is regarding the partial regularity of suitable weak solutions up to the boundary.

\begin{theorem}
\label{th3}
Let $\Omega$ be a domain in $\mathbb{R}^4$ with uniform $C^2$ boundary and $f \in L_{6}(Q_T)$. Let $(u,p)$ be a suitable weak solution of \eqref{ns} in $Q_T$ with the boundary condition \eqref{bd}. Then up to the boundary, the 2D Hausdorff measure of the set of singular points is equal to zero.
\end{theorem}

We have similar results for the 6D stationary case. The notation in Theorems \ref{mainthm6d} and \ref{th26d} are introduced in Section \ref{6ds1}.
\begin{theorem}
\label{mainthm6d}
Let $\Omega$ be a domain in $\mathbb{R}^6$ and $f \in L_{6}(\Omega)$.  Let $(u,p)$ be a suitable weak solution of \eqref{ns6d} in $\Omega$ with the boundary condition \eqref{bd}. There is a positive number $\epsilon_0$ satisfying the following property. Assume that for a point ${\hat x} \in \partial\Omega$, $\Omega({\hat x},R)=B^{+}({\hat x},R)$ for some small $R$ and the inequality
\begin{equation*}
\limsup_{r\searrow 0} E(r) \leq \epsilon_0
\end{equation*}
holds. Then $u$ is H\"older continuous near ${\hat x}$.
\end{theorem}

\begin{theorem}
\label{th26d}
Let $\Omega$ be a domain in $\mathbb{R}^6$ and $f \in L_{6}(\Omega)$. Let $(u,p)$ be a suitable weak solution of \eqref{ns6d} in $\Omega$ with the boundary condition \eqref{bd}.
There is a positive number $\epsilon_0$ satisfying the following property. Assume that for a point ${\hat x}\in \partial\Omega$, $\Omega({\hat x},R)=B^{+}({\hat x},R)$ for some small $R$, and for some $\rho_0 >0$ we have
\begin{equation*}
C(\rho_0)+D(\rho_0)+G(\rho_0) \leq \epsilon_0.
\end{equation*}
Then $u$ is H\"older continuous near ${\hat x}$.
\end{theorem}

\begin{theorem}
\label{th36d}
Let $\Omega$ be a domain in $\mathbb{R}^6$ with uniform $C^2$ boundary and $f \in L_{6}(\Omega)$. Let $(u,p)$ be a suitable weak solution of \eqref{ns6d} in $\Omega$ with the boundary condition \eqref{bd}. Then up to the boundary, the 2D Hausdorff measure of the set of singular points in $\Omega$  is equal to zero.
\end{theorem}

We give a unified approach for both cases.
Compared to \cite{Dong_12} and \cite{Dong_13}, the main obstacle in showing the partial regularity up to the boundary is that the control of the pressure associated with $u$ at the boundary is difficult. In the interior case, the pressure can be decomposed as
a sum of a harmonic function and a term which can be easily controlled in terms of $u$ by the Calder\'on--Zygmund estimate. However, such approach does not seem to be applicable near the boundary. This is because the boundary condition for $p$ is not prescribed, which makes it out of reach to estimate the harmonic function in the decomposition. Our proofs follow the scheme in \cite{Dong_12} and \cite{Dong_13} mentioned before, and use a decomposition of the pressure $p$ introduced by Seregin \cite{Seregin_02}. It turns out that, with Seregin's decomposition, we cannot get the initial weak decay estimate by using a simple manipulation of inequalities as in \cite{Dong_12} and \cite{Dong_13}. Roughly speaking, this is due to the fact that in some sense $p$ is corresponding to $|u|$ in the new decomposition, instead of $|u|^2$. Consequently, the estimate obtained for $p$ is not as sharp as in the interior case. To this end, we use a new iteration argument to establish an initial decay estimate. Although this argument looks more involved, it is in fact more flexible and also enables us to get rid of the use of an additional scale-invariant quantity defined in \cite{Dong_13}.

We remark that by using the same method we can get an alternative proof of Seregin or Kang's results without using any compactness argument. It remains an interesting open problem whether a similar result can be obtained for higher dimensional Navier--Stokes equations ($d\geq 7$ in the stationary case and $d\geq 5$ in the time-dependent case).  It seems to
us that four is the highest dimension for the time-dependent case and six is the highest dimension for the stationary case to which our approach (or any
existing approach) applies. In fact, for the time-dependent case,  by the embedding theorem
$$
L_{\infty}((0,T);L_2(\Omega))\cap L_2((0,T);W_2^1(\Omega))\hookrightarrow L_{2(d+2)/d}((0,T)\times \Omega),
$$
which implies nonlinear term in the energy inequality cannot
be controlled by the energy norm when $d \ge 5$.

This paper is divided into two parts: the 4D time-dependent case (Section \ref{4d}) and the 6D stationary case (Section \ref{6d}). We organize each part as follows: In Section \ref{s1}, we introduce the notation of certain scale-invariant quantities and some settings which are used throughout the Section \ref{4d}. In Section \ref{s2}, we prove our results in three steps. In the first step, we give some estimates of the scale-invariant quantities, which are by now standard and essentially follow the arguments in \cite{Dong_07, Lin_98}. In the second step, we establish a weak decay estimate of certain scale-invariant quantities by using an iteration argument based on the estimates we proved in the first step. In the last step, we successively improve the decay estimate by a bootstrap argument, and apply parabolic regularity to get a good estimate of $L_{3/2}$-mean oscillations of $u$, which yields the H\"{o}lder continuity of $u$ according to Campanato's characterization of H\"older continuous functions. Since the 6D stationary case is similar by using this approach, we treat it briefly in Section \ref{6d}.

\section{4D time-dependent case}\label{4d}
\subsection{Notation and Settings}\label{s1}

In this section, we introduce the notation which are used throughout Section \ref{4d}.  Let $\Omega$ be a domain in $\bR^4$, $-\infty\le S<T\le \infty$, and $m,n\in [1,\infty]$. We denote $L_{m,n}(\Omega\times (S,T))$ to be the usual space-time Lebesgue spaces of functions with the norm
\begin{align*}
\|f\|_{L_{m,n}(\Omega\times (S,T))}=\Big(\int_S^T\|f\|_{L_m(\Omega)}^n\,dt\Big)^{1/n}\quad\text{for}\quad n <+\infty,\\
\|f\|_{L_{m,n}(\Omega\times (S,T))}=\esssup_{t\in(S,T)}\|f\|_{L_m(\Omega)}\quad\text{for}\quad n = +\infty.
\end{align*}
We will also use the following Sobolev spaces:
\begin{align*}
W_{m,n}^{1,0}(\Omega\times (S,T))&=\Big\{f\,\Big|\, \|f\|_{L_{m,n}(\Omega\times (S,T))}+\|\nabla f\|_{L_{m,n}(\Omega\times (S,T))} < +\infty\Big\},\\
W_{m,n}^{2,1}(\Omega\times (S,T))&=\Big\{f\,\Big|\, \|f\|_{L_{m,n}(\Omega\times (S,T))}+\|\nabla f\|_{L_{m,n}(\Omega\times (S,T))}\\&\quad+\|\nabla^2 f\|_{L_{m,n}(\Omega\times (S,T))}+\|\partial_t f\|_{L_{m,n}(\Omega\times (S,T))} < +\infty\Big\}.
\end{align*}
For $p\in (1,\infty)$, we denote $\cH_p^1$ to be the solution spaces for divergence form parabolic equations. Precisely,
$$
\cH^1_p(\Omega\times (S,T) )=
\{u\,|\, u,Du \in L_p(\Omega\times (S,T) ),\,u_t \in \bH^{-1}_p(\Omega\times (S,T) \},
$$
where $\bH^{-1}_p(\Omega\times (S,T) )$ is the space consisting of all generalized functions $v$ satisfying
$$
\inf \big\{\|f\|_{L_p(\Omega\times (S,T)) }+\|g\|_{L_p(\Omega\times (S,T)) }\,|\,v=\nabla\cdot g+f\big\}<\infty.
$$

We shall use the following notation of balls, half balls, spheres, half spheres, parabolic cylinders, half parabolic cylinders:
\begin{align*}
&B({\hat x},r)=\{x\in\mathbb{R}^4\,|\,|x-{\hat x}|<r\},\quad B(r)=B(0,r),\quad B=B(1);\\
&B^{+}({\hat x},r)=\{x\in B({\hat x},r)\,|\,x=(x',x_4),\,x_4>\hat x_{4}\},\\&B^{+}(r)=B^{+}(0,r),\quad B^{+}=B^{+}(1);\\
&S^+({\hat x},r)=\{x\in \bR^4\,|\,|x-{\hat x}|=r,\,x=(x',x_4),\,x_4>\hat x_{4} \};\\
&Q({\hat z},r)=B({\hat x},r) \times ({\hat t}-r^2,{\hat t}),\quad Q(r)=Q(0,r),\quad Q=Q(1);\\
&Q^{+}({\hat x},r)=B^{+}({\hat x},r) \times ({\hat t}-r^2,{\hat t}),\quad Q^{+}(r)=Q^{+}(0,r),\quad Q^{+}=Q^{+}(1),
\end{align*}
where ${\hat z}=({\hat x},{\hat t})$ and $\hat x_{4}$ is the fourth coordinate of ${\hat x}$. We also define
$$
\Omega({\hat x},r)=B({\hat x},r)\cap \Omega,\quad \omega({\hat z},r)=Q({\hat z},r)\cap Q_T.
$$

We denote mean values of summable functions as follows:
\begin{align*}
[u]_{{\hat x},r}(t)&=\dfrac{1}{|\Omega(\hat x,r)|}\int_{\Omega({\hat x},r)}u(x,t)\,dx,\\
(u)_{{\hat z},r}&=\dfrac{1}{|\omega(\hat z,r)|}\int_{\omega({\hat z},r)}u\, dz,
\end{align*}
where $|A|$ as usual denotes the Lebesgue measure of the set $A$.


Now we introduce the following important quantities:
\begin{align*}
A(r)&=A(r,{\hat z})=\esssup_{{\hat t}-r^2\leq t\leq {\hat t}}\dfrac{1}{r^2}\int_{\Omega({\hat x},r)}|u(x,t)|^2\,dx,\\
E(r)&=E(r,{\hat z})=\dfrac{1}{r^2}\int_{\omega({\hat z},r)}|\nabla u|^2\,dz,\\
C(r)&=C(r,{\hat z})=\dfrac{1}{r^3}\int_{\omega({\hat z},r)}|u|^3\,dz,\\
D(r)&=D(r,{\hat z})=\dfrac{1}{r^3}\int_{\omega({\hat z},r)}|p-[p]_{{\hat x},r}|^{3/2}\,dz,\\
G(r)&=G(r,{\hat z})=r^4\Big[\int_{\omega({\hat z},r)}|f|^6\,dz\Big]^{1/3}.
\end{align*}
Notice that all these quantities are invariant under the natural scaling:
$$
u_\lambda(x,t)=\lambda u(\lambda x,\lambda^2 t), \,\,
p_\lambda(x,t)=\lambda^2 p(\lambda x,\lambda^2 t),\,\,
f_\lambda(x,t)=\lambda^3 f(\lambda x,\lambda^2 t).
$$
We shall estimate them in Section \ref{s2}. 


\subsection{The proof}\label{s2}
In the proofs below, we shall make use of the following well-known interpolation inequality.
\begin{lemma}
\label{interpolation}
For any functions $u \in W_2^1(\mathbb{R}^4_{+})$  and real numbers $q\in [2,4]$ and $r >0$,
\begin{align*}
\int_{B^{+}(r)}|u|^q \,dx &\leq N(q)\Big[\big(\int_{B^{+}(r)}|\nabla u|^2 \, dx\big)^{q-2}\big(\int_{B^{+}(r)}|u|^2 \, dx\big)^{2-q/2}\\
&\quad+r^{-2(q-2)}\big(\int_{B^{+}(r)}|u|^2\,dx\big)^{q/2}\Big].
\end{align*}
\end{lemma}

Let $\cL:=\partial_t-\partial_{x_i}(a_{ij}\partial_{x_j})$ be a (possibly degenerate) divergence form parabolic operator with measurable coefficients which are bounded by a constant $K>0$. We shall use the following Poincar\'e type inequality for solutions to parabolic equations. See, for instance, \cite[Lemma 3.1]{Krylov_05}.
\begin{lemma}
                    \label{lem11.31}
Let ${\hat z}\in \bR^{d+1}$, $p\in (1,\infty)$, $r\in (0,\infty)$, $u\in \cH^1_{p}(Q^{+}({\hat z},r))$, $g=(g_1,\ldots,g_d),f\in L_{p}(Q^{+}({\hat z},r))$. Suppose that $u$ is a weak solution to $\cL u=\nabla\cdot g+f$ in $Q^{+}({\hat z},r)$. Then we have
$$
\int_{Q^{+}({\hat z},r)}|u(t,x)-(u)_{{\hat z},r}|^p\,dz\le Nr^p
\int_{Q^{+}({\hat z},r)}\big(|\nabla u|^p+|g|^p+r^p|f|^p\big)\,dz,
$$
where $N=N(d,K,p)$.
\end{lemma}

Lastly, we recall the following two important lemmas which will be used to handle the estimates for the pressure $p$.
\begin{lemma}
\label{lest_1}
Let $\Omega \subset \bR^4$ be a bounded domain with smooth boundary and $T >0$ be a constant. Let $1< m <+\infty, 1< n <+\infty$ be two fixed numbers. Assume that $g \in L_{m,n}(Q_T)$. Then there exists a unique function pair $(v,p)$, which satisfies the following equations:
$$
\left\{\begin{aligned}
\partial_t v-\Delta v+\nabla p &= g \quad \text{in}\quad Q_T,\\
\nabla \cdot v&=0\quad\text{in}\quad Q_T,\\
[p]_{\Omega}(t)&=0\quad\text{for}\quad\text{a.e.}\,\, t\in [0,T],\\
v&=0 \quad\text{on}\quad \partial_p Q_T.
\end{aligned}
\right.
$$
Moreover, $v$ and $p$ satisfy the following estimate:
\begin{align*}
\|v\|_{W_{m,n}^{2,1}(Q_T)}+\|p\|_{W_{m,n}^{1,0}(Q_T)} \leq C\|g\|_{L_{m,n}(Q_T)},
\end{align*}
where the constant $C$ only depends on $m$, $n$, $T$, and $\Omega$.
\end{lemma}

\begin{lemma}
\label{lest_2}
Let $1< m \leq 2$, $1< n \leq 2$, and $m\leq s <+\infty$ be constants and $g\in L_{s,n}(Q^+)$.
Assume that the functions $v \in W^{1,0}_{m,n}(Q^+)$ and $p\in L_{m,n}(Q^+)$ satisfy the equations£º
$$
\left\{\begin{aligned}
\partial_t v-\Delta v+\nabla p &= g \quad\text{in}\quad Q^{+},\\
\nabla \cdot v&=0\quad\text{in}\quad Q^{+},
\end{aligned}
\right.
$$
and the boundary condition
\begin{equation*}
v=0 \quad\text{on}\quad \{y\, |\, y=(y',0),|y'|<1\} \times [-1,0).
\end{equation*}
Then, we have $v\in W^{2,1}_{s,n}(Q^+(1/2))$, $p\in W^{1,0}_{s,n}(Q^+(1/2))$, and
\begin{align*}
&\|v\|_{W_{s,n}^{2,1}(Q^{+}(1/2))}+\|p\|_{W_{s,n}^{1,0}(Q^{+}(1/2)}\\
&\, \leq C\big(\|g\|_{L_{s,n}(Q^{+})}
+\|v\|_{W_{m,n}^{1,0}(Q^{+})}+\|p\|_{L_{m,n}(Q^{+})}\big),
\end{align*}
where the constant $C$ only depends on $m$, $n$, and $s$.
\end{lemma}

We refer the reader to \cite{MaSo94} for the proof of Lemma \ref{lest_1}, and \cite{Seregin_03, Seregin_09} for the proof of Lemma \ref{lest_2}.

\begin{remark}
In \cite{Seregin_03}, Lemma \ref{lest_2} was proved under the stronger conditions that $v\in W^{2,1}_{m,n}(Q^+)$ and $p\in W^{1,0}_{m,n}(Q^+)$. See Proposition 2 there. These conditions were relaxed in \cite{Seregin_09} by mollifying the functions with respect to $x'$ and passing to the limit. In fact, the proof can be simplified if the mollification is taken with respect to both $x'$ and $t$. Indeed, let $v^{\varepsilon}$, $p^{\varepsilon}$, and $g^{\varepsilon}$ the standard mollification with respect to $(t,x')$, which satisfy the same equations as $v$, $p$, and $g$. It is clear that for sufficiently small $\varepsilon$
$$
DD_{x'} v^{\varepsilon},\,\, \partial_t v^{\varepsilon},\,\, D_{x'}p^{\varepsilon}\in L_{m,n}(Q^+(3/4)).
$$
Then from the equations for $v_1^{\varepsilon},\ldots,v_3^{\varepsilon}$, we get $D_{x_4x_4}v_j^{\varepsilon}\in L_{m,n}(Q^+(3/4))$ for $j=1,2,3$.
Owing to $\nabla\cdot v^{\varepsilon}=0$,  $D_{x_4x_4}v_4^{\varepsilon}\in L_{m,n}(Q^+(3/4))$, which together with the equation for $v_4^{\varepsilon}$ further implies $D_{x_4}p^{\varepsilon}\in L_{m,n}(Q^+(3/4))$. By Proposition 2 of \cite{Seregin_03}, we have
$v^{\varepsilon}\in W^{2,1}_{s,n}(Q^+(1/2))$, $p^{\varepsilon}\in W^{1,0}_{s,n}(Q^+(1/2))$, and
\begin{align*}
&\|v^{\varepsilon}\|_{W_{s,n}^{2,1}
(Q^{+}(1/2))}+\|p^{\varepsilon}\|_{W_{s,n}^{1,0}
(Q^{+}(1/2))}\\
&\, \leq C\big(\|g^{\varepsilon}\|_{L_{s,n}(Q^{+}(3/4))}
+\|v^{\varepsilon}\|_{W_{m,n}^{1,0}(Q^{+}(3/4))}
+\|p^{\varepsilon}\|_{L_{m,n}(Q^{+}(3/4))}\big),
\end{align*}
where $C$ is independent of $\varepsilon$. Taking the limit as $\varepsilon\to 0$, the conclusion of Lemma \ref{lest_2} follows.
\end{remark}

Now we prove the main theorems in three steps.
\subsubsection{Step 1.}
First, we control the quantities $A$, $C$, and $D$ in a smaller ball by their values in a larger ball under the assumption that $E$ is sufficiently small. Here we follow the argument in \cite{Dong_07}, which in turn used some ideas in \cite{Lady_99, Lin_98, Seregin_02}.

\begin{lemma}
                                    \label{lem2.5}
Let $\gamma \in (0,1)$ and $\rho >0$ be constants. Suppose that ${\hat x}\in \partial\Omega$ and $\omega({\hat z},\rho)=Q^{+}({\hat z},\rho)$ so that $\partial\Omega$ is locally flat near ${\hat x}$. Then we have
\begin{equation}
\label{C_est}
C(\gamma \rho) \leq N[\gamma^{-3}A^{1/2}(\rho)E(\rho)+\gamma^{-9/2}A^{3/4}(\rho)E^{3/4}(\rho)+\gamma C(\rho)],
\end{equation}
where $N$ is a constant independent of $\gamma$, $\rho$, and ${\hat z}$.
\end{lemma}
\begin{proof}
This is Lemma 2.8 of \cite{Dong_07} with the only difference that balls (or cylinders) are replaced by half balls (or half cylinders, respectively). By using the zero boundary condition, the proof remains the same with obvious modifications. We omit the details.
\end{proof}

\begin{lemma}
\label{D_est_lemma}
Let $ \gamma \in (0,1/4]$ and $\rho >0$ be constants. Suppose that ${\hat x}\in \partial\Omega$ and $\omega({\hat z},\rho)=Q^{+}({\hat z},\rho)$. Then we have
\begin{align}
\label{D_est}
\nonumber D(\gamma \rho) &\leq N\big[\gamma^{-3}A^{1/2}(\rho)E(\rho)
+\gamma^{9/4}(D(\rho)\\
&\quad+A^{3/4}(\rho)+E^{3/4}(\rho))+\gamma^{-3}G^{3/4}(\rho)\big],
\end{align}
where $N$ is a constant independent of $\gamma$, $\rho$, and ${\hat z}$.
\end{lemma}

\begin{proof}
Without loss of generality, by shifting the coordinates we may assume that ${\hat z}=(0,0)$. By the scale-invariant property, we may also assume $\rho=1$.
We choose and fix a domain $\tilde{B} \subset \mathbb{R}^{4}$ with smooth boundary so that
\begin{equation*}
B^{+}(1/2) \subset \tilde{B} \subset B^{+},
\end{equation*}
and denote $\tilde{Q}=\tilde{B}\times(-1,0)$.
Define $\tilde{f}=- u\cdot\nabla u$. By using H\"{o}lder's inequality, Lemma \ref{interpolation} with $q=12/5$ and the boundary Poincar\'e inequality, we get
\begin{align}
    \label{eq5.37}
\nonumber&\Big(\int_{B^{+}}|\tilde f|^{12/11}\,dx\Big)^{11/8}\\\nonumber &\,\leq\Big(\int_{B^{+}}|\nabla u|^2\,dx\Big)^{3/4}\Big(\int_{B^{+}}|u|^{12/5}\,dx\Big)^{5/8}\\
\nonumber &\,\leq N\Big(\int_{B^{+}}|\nabla u|^2\,dx\Big)^{3/4}\Big(\int_{B^{+}}|\nabla u|^{2}\,dx\Big)^{1/4}\Big(\int_{B^{+}}|u|^{2}\,dx\Big)^{1/2}\\
&\,\leq N\Big(\int_{B^{+}}|\nabla u|^2\,dx\Big)\Big(\int_{B^{+}}|u|^2\,dx\Big)^{1/2},
\end{align}
and
\begin{align}
\label{v1p1_3}
\|f\|_{L_{\frac{12}{11},\frac 3 2}(\tilde{Q})}\leq N\Big(\int_{\tilde{Q}}|f|^6\,dz\Big)^{1/6}.
\end{align}
By Lemma \ref{lest_1}, there is a unique solution
$$
v\in W^{2,1}_{12/11,3/2}(\tilde{Q})\quad \text{and}\quad
p_1\in W^{1,0}_{12/11,3/2}(\tilde{Q})
$$
to the following initial boundary value problem:
$$
\left\{\begin{aligned}
\partial_t v-\Delta v+\nabla p_1 &= \tilde{f}+f \quad\text{in}\quad \tilde{Q},\\
\nabla \cdot v&=0\quad\text{in}\quad \tilde{Q},\\
[p_1]_{\tilde{B}}(t)&=0\quad \text{for}\quad \text{a.e.}\,\, t\in(-1,0),\\
v&=0 \quad\text{on}\quad \partial_p  \tilde{Q}.
\end{aligned}
\right.
$$
Moreover, we have
\begin{align}
\label{v1p1}
\nonumber &\|v\|_{L_{\frac{12}{11},\frac 3 2}(\tilde{Q})}+\|\nabla v\|_{L_{\frac{12}{11},\frac 3 2}(\tilde{Q})}+\|p_1\|_{L_{\frac{12}{11},\frac 3 2}(\tilde{Q})}+\|\nabla p_1\|_{L_{\frac{12}{11},\frac 3 2}(\tilde{Q})}\\ \nonumber
&\,\leq N\|\tilde{f}\|_{L_{\frac {12}{11},\frac 3 2}(\tilde{Q})}+
N\|f\|_{L_{\frac{12}{11},\frac 3 2}(\tilde{Q})}\\
&\,\le N\Big(\int_{-1}^{0}
\big(\int_{B^{+}}|\nabla u|^2\,dx\big)\big(\int_{B^{+}}|u|^2\,dx\big)^{1/2}\,dt
\Big)^{2/3}+N\Big(\int_{\tilde{Q}}
|f|^6\,dz\Big)^{1/6},
\end{align}
where in the last inequality we used \eqref{eq5.37} and \eqref{v1p1_3}.

We set $w=u-v$ and $p_2=p-p_1-[p]_{0,1/2}$. Then $w$ and $p_2$ satisfy
$$
\left\{\begin{aligned}
\partial_t w-\Delta w +\nabla p_2 &=0\quad\text{in}\quad \tilde{Q},\\
\nabla \cdot w&=0\quad\text{in}\quad \tilde{Q},\\
w&=0 \quad\text{on}\quad \big\{\partial \tilde{B} \cap \partial\Omega\big\}\times[-1,0).
\end{aligned}
\right.
$$
By Lemma \ref{lest_2} together with a scaling and the triangle inequality, we have $p_2\in W^{1,0}_{24,3/2}(Q^{+}(1/4))$ and
\begin{align}
\nonumber&\|\nabla p_2\|_{L_{24,\frac 3 2}(Q^{+}(1/4))}\\
\nonumber &\,\leq N\Big[\|w\|_{L_{\frac {12} {11},\frac 3 2}(Q^{+}(1/2))}+\|\nabla w\|_{L_{\frac {12} {11},\frac 3 2}(Q^{+}(1/2))}+
\|p_2\|_{L_{\frac {12} {11},\frac 3 2}(Q^{+}(1/2))}\Big]\\
\nonumber&\,\leq N\Big[\|u\|_{L_{\frac {12}{11},\frac 3 2}(Q^{+}(1/2))}+\|\nabla u\|_{L_{\frac {12}{11},\frac 3 2}(Q^{+}(1/2))}\\
\nonumber&\,\quad+\|p-[p]_{0,1/2}\|_{L_{\frac {12}{11},\frac 3 2}(Q^{+}(1/2))}
+\|v\|_{L_{\frac {12}{11},\frac 3 2}(Q^{+}(1/2))}
\\&\,\quad+\|\nabla v\|_{L_{\frac {12}{11},\frac 3 2}(Q^{+}(1/2))}+
\|p_1\|_{L_{\frac {12}{11},\frac 3 2}(Q^{+}(1/2))}\Big].
\label{p2}
\end{align}
Here the constant $s=24$ is non-essential and can be replaced by any sufficiently larger number.
Then with \eqref{v1p1} and H\"older's inequality, we obtain
\begin{align}
                            \label{eq5.43}
&\nonumber \|\nabla p_2\|_{{L_{24,\frac 3 2}(Q^{+}(1/4))}}\\
\nonumber&\,\leq  N\Big[\|u\|_{L_{\frac {12} {11},\frac 3 2}(Q^{+}(1/2))}+\|\nabla u\|_{L_{\frac {12} {11},\frac 3 2}(Q^{+}(1/ 2))}
+\|p-[p]_{0,1/2}\|_{L_{\frac 3 2,\frac 3 2}(Q^{+}(1/ 2))}\\
&\,\quad+\big(\int_{-1}^{0}
(\int_{B^{+}}|\nabla u|^2\,dx)(\int_{B^{+}}|u|^2\,dx)^{1/2}\,dt
\big)^{2/3}+(\int_{\tilde{Q}}
|f|^6\,dz)^{1/6}\Big].
\end{align}
Recall that $0 < \gamma \leq 1 /4$. Then by using the Sobolev--Poincar\'e inequality, the triangle inequality, \eqref{v1p1}, \eqref{eq5.43}, 
and H\"{o}lder's inequality, we bound $D(\gamma)$ by
\begin{align*}
&\dfrac{N}{\gamma^3}\int_{-\gamma^2}^{0}\big(\int_{B^{+}(\gamma)} |\nabla p_1|^{\frac {12} {11}}\,dx\big)^{\frac {11} 8}+\big(\int_{B^{+}(\gamma)} |\nabla p_2|^{\frac {12} {11}}\,dx\big)^{\frac {11} 8}\,dt\\
\nonumber&\,\leq N\big[\gamma^{-3}E(1)A^{\frac 1 2}(1)+\gamma^{-3}G^{\frac 3 4}(1) \big] + N\gamma^{\frac 9 4}\int_{-\gamma^2}^{0}\big(\int_{B^{+}(\gamma)} |\nabla p_2|^{24}\,dx\big)^{\frac 1 {16}}\,dt\\\nonumber&\,\leq N\big[\gamma^{-3}E(1)A^{\frac 1 2}(1)+\gamma^{-3}G^{\frac 3 4}(1)\big]+ N\gamma^{\frac 9 4}[E(1)A^{\frac 1 2}(1)+G^{\frac 3 4}(1)\\
\nonumber &\quad+D(1)+A^{\frac 3 4}(1)+E^{\frac 3 4}(1)]
\\
\nonumber&\,\leq N\Big[\gamma^{-3}E(1)A^{\frac 1 2}(1)+\gamma^{\frac 9 4}\big(D(1)+A^{\frac 3 4}(1)+E^{\frac 3 4}(1)\big)
+\gamma^{-3}G^{\frac 3 4}(1)\Big].
\end{align*}
The lemma is proved.
\end{proof}

The proofs of Lemma \ref{aelemma} and Proposition \ref{prop1} as well as Lemma \ref{eeeeee} in Section \ref{sec2.2.2} follow closely the argument in \cite{Dong_07}. We include them in Appendix for completeness.

\begin{lemma}
                \label{aelemma}
Let $\gamma \in (0,1/2]$ and $\rho >0$ be constants. Suppose that ${\hat x}\in \partial\Omega$ and $\omega({\hat z},\rho)=Q^{+}({\hat z},\rho)$. Then we have
\begin{equation*}
A(\gamma \rho)+E(\gamma \rho) \leq N \gamma^{-2}\big[C^{2/3}(\rho)+C(\rho)+C^{1/3}(\rho)D^{2/3}(\rho)+G(\rho)\big].
\end{equation*}
In particular, when $\gamma=1/2$ we have
\begin{equation}
\label{AE_est_2}
A(\rho/2)+E(\rho/2) \leq N \big[C^{2/3}(\rho)+C(\rho)+C^{1/3}(\rho)D^{2/3}(\rho)+G(\rho)\big].
\end{equation}
\end{lemma}

As a conclusion, we obtain
\begin{proposition}
\label{prop1}
For any $\epsilon_0 >0$, there exists $\epsilon_1 >0$ small such that the following is true. For any ${\hat z}=({\hat x},{\hat t}) ,{\hat x}\in \partial\Omega$ satisfying $\omega({\hat z},R)=Q^{+}({\hat z},R)$ for some small $R$ and
\begin{equation}
\limsup_{r \searrow 0} E(r)\leq \epsilon_1,
\label{e_cond}
\end{equation}
we can find $\rho_0$ sufficiently small such that
\begin{equation*}
A(\rho_0)+E(\rho_0)+C(\rho_0)+D(\rho_0) \leq \epsilon_0.
\end{equation*}

\end{proposition}

\subsubsection{Step 2.}
                                    \label{sec2.2.2}
In the second step, first we estimate the values of $A$, $E$, $C$, and $D$ in a smaller ball by the values of themselves in a larger ball.
\begin{lemma}
\label{eeeeee}
Let $\rho >0$ and $\gamma \in (0,1/8]$ be constants. Suppose that ${\hat x}\in \partial\Omega$ and $\omega({\hat z},\rho)=Q^{+}({\hat z},\rho)$. Then we have
\begin{align}
\label{eree}
\nonumber&A(\gamma \rho)+E(\gamma \rho)\\
&\,\leq N \gamma^2A(\rho)+N\gamma^{-3}\big[C(\rho)+C^{1/3}(\rho)D^{2/3}(\rho)\big]+N\gamma^{-6}G(\rho),
\end{align}
where $N$ is a constant independent of $\rho$, $\gamma$, and ${\hat z}$.
\end{lemma}

In the next proposition, we study the decay property of $A$, $C$, $E$, and $D$ as the radius $\rho$ goes to zero.
\begin{proposition}
\label{prop2}
There exists $\epsilon_0>0$ satisfying the following property. Suppose that for some ${\hat z}=({\hat x},{\hat t})$, where ${\hat x}\in \partial\Omega$, and $\rho_0 >0$, it holds that $\omega({\hat z},\rho_0)=Q^{+}({\hat z},\rho_0)$ and
\begin{equation}
C(\rho_0)+D(\rho_0)+G(\rho_0)\leq \epsilon_0.
\label{cond_cdf}
\end{equation}
Then we can find $N>0$ and $\alpha_0 \in (0,1)$ such that for any $\rho \in (0,\rho_0/4)$ and $z^* \in Q({\hat z},\rho_0/4)\cap (\partial\Omega\times({\hat t}-\rho_0^2/16,{\hat t}))$, the following inequality holds uniformly
\begin{equation}
A(\rho,z^*)+C^{2/3}(\rho,z^*)+E(\rho,z^*)+ D(\rho,z^*)\leq N \rho ^{\alpha_0},
\label{cdf}
\end{equation}
where $N$ is a positive constant independent of $\rho$ and $z^*$.
\end{proposition}
\begin{proof}
Let $\epsilon'>0$ be a small constant to be specified later. Due to \eqref{AE_est_2} and \eqref{cond_cdf}, one can find $\epsilon_0=\epsilon_0(\epsilon')>0$ sufficiently small such that,
\begin{equation*}
\epsilon_0<\epsilon'^4,\quad A(\rho_0/2)+E(\rho_0/2)\leq \epsilon',\quad D(\rho_0/2) \leq \epsilon'.
\end{equation*}
Without loss of generality, we can assume that $\rho_0=\epsilon'$. If $\rho_0\neq \epsilon'$, since $C$, $D$, and $G$ are invariant under the natural scaling, we get \eqref{cdf} with $N$ proportional to $\rho_0^{-\alpha_0}$ after a scaling.

Owing to Lemma \ref{interpolation} with $q=3$, we get
\begin{equation}
C(\rho)\leq N\big[A(\rho)+E(\rho)\big]^{3/2},
\label{hhh}
\end{equation}
which implies
\begin{equation*}
C(\rho_0/2)\leq N\epsilon'^{3/2}.
\end{equation*}

For any $z^* \in Q({\hat z},\rho_0/4)\cap (\partial\Omega\times({\hat t}-\rho_0^2/16,{\hat t}))$, by using
\begin{equation*}
Q^{+}(z^*,\rho_0/4)\subset Q^{+}({\hat z},\rho_0/2)\subset Q_T,
\end{equation*}
we get
\begin{equation*}
A(\rho_1,z^*)+E(\rho_1,z^*)+C^{2/3}(\rho_1,z^*)+D(\rho_1,z^*) \leq N\epsilon'
\end{equation*}
with $\rho_1=\rho_0/4$.

Next, we shall prove inductively that
\begin{equation}
A(\rho_k,z^*)+E(\rho_k,z^*)+C^{2/3}(\rho_k,z^*)\le \rho_k^{\frac{1}{10}},
\quad D(\rho_k,z^*) \leq \rho_k^{\frac{1}{10}},
\label{pop}
\end{equation} where $\rho_k=\rho_1^{(1+\beta)^k}$ and $\beta=\frac{1}{200}$ for $k=1,2,\cdots$.

It is easy to see that \eqref{pop} holds for $k=1,2,3$ by choosing $\epsilon'$ sufficiently small.
Suppose that \eqref{pop} holds for $k\ge 3$. Since $\rho_{k+1}=\rho_{k}^{1+\beta}$, by using \eqref{eree} with $\gamma=\rho_k^\beta$ and $\rho=\rho_k$, we have

\begin{align}
\nonumber A(\rho_{k+1})+E(\rho_{k+1}) &\leq N \rho_{k}^{2\beta}A(\rho_k)+N \rho_k^{-3\beta}(C(\rho_k)+C^{1/3}(\rho_k)D^{2/3}(\rho_k))\\\nonumber&\quad+N\rho_k^{-6\beta}G(\rho_k)\\&\leq N \rho_k^{2\beta+\frac{1}{10}}+N\rho_k^{-3\beta+\frac{7}{60}}+N\rho_k^{-6\beta+4}.\label{ind1}
\end{align}
Here we used, for any $\rho\le \rho_0$,
\begin{equation}
                            \label{11.30}
G(\rho) \leq \rho^{4}\|f\|^2_{L^{6}(Q^{+}({\hat z},\rho))}
\leq  \rho^{4}\|f\|^2_{L^{6}(Q^{+}({\hat z},\rho_0))}
\le \rho^4\epsilon'^{-4}\epsilon_0\leq \rho^4,
\end{equation}
which follows from the definition of $G$ and the choice of $\epsilon_0$.
Since
$$
\min\{2\beta+\frac{1}{10},-3\beta+\frac{7}{60},-6\beta+4\}>\frac{1}{10}(1+\beta),
$$
we have
\begin{align*}
A(\rho_{k+1})+E(\rho_{k+1})\leq N\rho_{k+1}^{\frac{1}{10}+\xi}\,\, \text{for some}\,\, \xi>0,
\end{align*}
and by \eqref{hhh},
\begin{equation*}
C(\rho_{k+1})\leq N\rho_{k+1}^{\frac{3}{20}+\frac{3}{2}\xi}.
\end{equation*}
Here $N$ is a constant independent of $k$ and $\xi$.
By taking $\epsilon'$ sufficiently small that $N\epsilon'^{\xi} < 1$, we obtain
\begin{align}
\label{ind2}
A(\rho_{k+1})+E(\rho_{k+1})+C^{2/3}(\rho_{k+1})\leq \rho_{k+1}^{\frac{1}{10}}.
\end{align}

To estimate the remaining term $D(\rho_{k+1})$,
we apply Lemma \ref{D_est_lemma}. It turns out that, different from above, using the estimates of $A(\rho_{k})$, $E(\rho_k)$, and $D(\rho_k)$, one cannot get the estimate of $D(\rho_{k+1})$ as claimed. Instead, we shall bound $D(\rho_{k+1})$ by using the estimates which we get in the $k-2$-th step. By defining $\tilde \beta=(1+\beta)^3-1$ and using \eqref{D_est} with $\gamma=\rho_{k-2}^{\tilde \beta}$ and $\rho=\rho_{k-2}$, we get
\begin{align*}
D(\rho_{k+1})&\leq N \rho_{k-2}^{-3{\tilde \beta}}E(\rho_{k-2})A^{1/2}(\rho_{k-2})+ N\rho_{k-2}^{9\tilde \beta/4}D(\rho_{k-2})\\&\quad+N\rho_{k-2}^{9\tilde \beta/4}\big(A^{3/4}(\rho_{k-2})+E^{3/4}(\rho_{k-2})\big)+N \rho_{k-2}^{-3\tilde \beta}G^{3/4}(\rho_{k-2})\\&\leq N\rho_{k-2}^{-3\tilde \beta+\frac{3}{20}}+N\rho_{k-2}^{\frac{9}{4}\tilde \beta+\frac{1}{10}}+N\rho_{k-2}^{\frac{9}{4}\tilde \beta+\frac{3}{40}}+N\rho_{k-2}^{-3\tilde \beta+3}.
\end{align*}
Since $\min\{-3\tilde \beta+\frac{3}{20},\frac{9}{4}\tilde \beta+\frac{1}{10},\frac{9}{4}\tilde \beta+\frac{3}{40},-3\tilde \beta+3\}>\frac{1}{10}(1+\beta)^3$, we have
\begin{align*}
D(\rho_{k+1})\leq \rho_{k+1}^{\frac{1}{10}}
\end{align*}
by taking $\epsilon'$ sufficiently small, but independent of $k$.

Now for any $\rho \in (0,\rho_0/4)$, we can find a positive integer $k$ such that $\rho_{k+1}\leq \rho < \rho_k$. Therefore,
\begin{align*}
&A(\rho)+E(\rho)+C^{2/3}(\rho)+D(\rho)\\
&\le \rho_k^3 \rho_{k+1}^{-3}\big(A(\rho_k)+E(\rho_k)+C^{2/3}(\rho_k)+D(\rho_k)\big)\\
&\le 2\rho_k^{\frac 1 {10}-3\beta}\leq 2 \rho^{\frac{1}{1+\beta}(\frac 1 {10}-3\beta)}.
\end{align*}
By choosing $\alpha_0=\frac{1}{1+\beta}(\frac 1 {10}-3\beta)$, the lemma is proved.
\end{proof}

\subsubsection{Step 3}
\label{boot}
In the final step, we shall use a bootstrap argument to successively improve the decay estimate \eqref{cdf}. However, as we will show below, the bootstrap argument itself only gives the decay of $E(\rho)$ no more than $\rho^{2}$, which is not enough for the H\"{o}lder regularity of $u$ since the spatial dimension is four (so that we need the decay exponent $2+\delta$ according to Campanato's characterization of H\"older continuous functions). We shall use parabolic regularity to fill in this gap.

First we prove Theorem \ref{th2}. We begin with the bootstrap argument. We will choose an increasing sequence of real numbers $\{\alpha_k\}_{k=1}^{m}\in (\alpha_0,2)$ so that, under the condition \eqref{cond_cdf} with ${\hat x}\in \partial\Omega$, the following estimates hold uniformly for all $\rho >0$ sufficiently small and $z^* \in Q({\hat z},\rho_0/4)\cap (\partial\Omega\times({\hat t}-\rho_0^2/16,{\hat t}))$ over the range of $\{\alpha_k\}_{k=0}^m$:
\begin{align}
\label{esti}
&A(\rho,z^*)+E(\rho,z^*)\leq N \rho^{\alpha_k},\quad C(\rho,z^*)\leq N\rho^{3\alpha_k/2},\\
\label{esti_D}
&D(\rho,z^*)\leq N\rho^{5\alpha_k/6}.
\end{align}
We prove this via iteration. The $k=0$ case for \eqref{esti} and \eqref{esti_D} was proved in \eqref{cdf} with a possibly different exponent $\alpha_0$.
Now suppose that \eqref{esti} and \eqref{esti_D} hold with the exponent $\alpha_k$. We first estimate $A(\rho,z^*)$ and $E(\rho,z^*)$. Let $\rho=\tilde{\gamma}\tilde{\rho}$ where $\tilde{\gamma}=\rho^{\mu}$, $\tilde{\rho}=\rho^{1-\mu}$ and $\mu \in (0,1)$ to be determined. We use \eqref{eree}, \eqref{esti}, and \eqref{11.30} to obtain
\begin{equation*}
A(\rho)+E(\rho) \leq N \rho^{2\mu}\rho^{\alpha_k(1-\mu)}+N \rho^{-3\mu}\rho^{\frac{19}{18}\alpha_k(1-\mu)}+N \rho^{4(1-\mu)}\rho^{-6\mu}.
\end{equation*}
Choose $\mu=\dfrac{\alpha_k}{90+\alpha_k}$. Then \eqref{esti} is proved for $A(\rho)+E(\rho)$ with the exponent
\begin{align*}
\alpha_{k+1}&:=\min\Big\{2\mu+\alpha_k(1-\mu),
\frac{19}{18}\alpha_k(1-\mu)-3\mu,4(1-\mu)-6\mu\Big\}\\
&=\dfrac{92}{90+\alpha_k}\alpha_k \in (\alpha_k,2).
\end{align*}
The estimate in \eqref{esti} with $\alpha_{k+1}$ in place of $\alpha_k$ for $C(\rho,z^*)$ follows from \eqref{hhh} immediately.
To prove the estimate in \eqref{esti_D} with $\alpha_{k+1}$, we use Lemma \ref{D_est_lemma}. Let $\rho=\tilde{\gamma}\tilde{\rho}$, where $\tilde{\gamma}=\rho^{\mu}$ and $\tilde{\rho}=\rho^{1-\mu}$ with a constant $\mu \in (0,1)$ to be specified. From \eqref{D_est}, \eqref{esti} with $\alpha_{k+1}$ in place of $\alpha_k$, \eqref{esti_D}, and \eqref{11.30}, we have
\begin{align*}
D(\rho) &\leq N\big[\rho^{-3\mu+\frac{3}{2}\alpha_{k+1}(1-\mu)}+\rho^{9\mu/4+\frac{5}{6}\alpha_{k}(1-\mu)}+\rho^{9\mu/4+\frac{3}{4}\alpha_{k+1}(1-\mu)}\\&\quad+\rho^{-3\mu+3(1-\mu)}\big].
\end{align*}
Choose $\mu=\dfrac{\alpha_{k+1}}{7+\alpha_{k+1}}$. Then we get
\begin{align*}
&\min\Big\{-3\mu+\frac{3}{2}\alpha_{k+1}(1-\mu),9\mu/4+\frac{5}{6}\alpha_{k}(1-\mu),9\mu/4+\frac{3}{4}\alpha_{k+1}(1-\mu),\\&\quad\quad -3\mu+3(1-\mu)\Big\}\\
&=\frac{15\alpha_{k+1}}{14+2\alpha_{k+1}},
\end{align*}
and
\begin{align*}
D(\rho)\leq N\rho^{\frac{15\alpha_{k+1}}{14+2\alpha_{k+1}}}\leq N\rho^{5\alpha_{k+1}/6}
\end{align*}
since $\alpha_{k+1}\in(0,2)$.
Moreover,
\begin{equation*}
2-\alpha_{k+1}=\dfrac{90}{90+\alpha_k}(2-\alpha_k) \leq \dfrac{90}{90+\alpha_0}(2-\alpha_k).
\end{equation*}
Thus, for any sufficiently small $\delta$, we can find a $m$ that $\alpha_m\geq 2-\delta$.

Via the bootstrap argument, we have got the following estimates for all $\rho >0$ sufficiently small and $z^*=(x^*,t^*) \in Q({\hat z},\rho_0/4)\cap (\partial\Omega\times({\hat t}-\rho_0^2/16,{\hat t}))$:
\begin{equation}
\sup_{t^*-\rho^2 \leq t\leq t^*}\int_{B^{+}(x^*,\rho)}|u(x,t)|^2\,dx \leq N \rho^{4-\delta},
\label{pp}
\end{equation}
\begin{equation}
\int_{Q^{+}(z^*,\rho)}|p-[p]_{x^*,\rho}|^{3/2} \,dz \leq N \rho^{3+\frac{5}{6}(2-\delta)},
\label{ppp}
\end{equation}
\begin{equation}
\int_{Q^{+}(z^*,\rho)}|u|^3 \,dz \leq N \rho^{3+\frac{3}{2}(2-\delta)}.
\label{pppp}
\end{equation}

Finally, we use the parabolic regularity theory to improve the decay estimate of mean oscillations of $u$ and then complete the proof.
We rewrite \eqref{ns} (in the weak sense) into
\begin{equation}
                                        \label{eq11.55}
\partial_t u_i- \Delta u_i=-\partial_j(u_iu_j)-\partial_i p+f_i.
\end{equation}
Due to \eqref{pp} and \eqref{pppp}, there exists $\rho_1 \in (\rho/2,\rho)$ such that
\begin{equation}
\begin{aligned}
\int_{B^{+}(x^*,\rho_1)}|u(x,t^*-\rho_1^2)|^2\,dx &\leq N \rho^{4-\delta},\\
\int_{t^*-\rho_1^2}^{t^*}\int_{S^{+}(x^*,\rho_1)}|u|^3 \,dx\,dt &\leq N \rho^{2+\frac{3}{2}(2-\delta)}.
\end{aligned}
\label{oo}
\end{equation}
Let $v$ be the unique weak solution to the heat equation
\begin{equation*}
\partial_t v-\Delta v=0 \ \ \text{in} \ \ Q^{+}(z^*,\rho_1)
\end{equation*}
with the boundary condition $v=u$ on $\partial_p Q^{+}(z^*,\rho_1)$. Since $v$ vanishes on the flat boundary part, it follows from the standard estimates for the heat equation, H\"{o}lder's inequality, and \eqref{oo} that
\begin{align}
 &\sup_{Q^{+}(z^*,\rho_1/2)}|\nabla v| \nonumber\\
 &\,\leq N \rho_1^{-6}\int_{t^*-\rho_1^2}^{t^*}\int_{S^{+}(x^*,\rho_1)}|v|\,  dx\,dt+N\rho_1^{-5}\int_{B^{+}(x^*,\rho_1)}|v(x,t^*-\rho_1^2)|\,dx\nonumber\\
 &\,\leq N \rho^{-1-\delta/2}.
\label{uu}
\end{align}
Denote $w=u-v.$ Then $w$ satisfies the inhomogeneous heat equation
\begin{equation*}
\partial_t w_i -\Delta w_i = -\partial_j(u_iu_j)-\partial_i (p-[p]_{x^*,\rho})+f_i\ \ \text{in} \ \ Q^{+}(z^*,\rho_1)
\end{equation*}
with the zero boundary condition. By the classical $L_p$ estimate for the heat equation, we have
\begin{align*}
\nonumber\|\nabla w\|_{L_{3/2}(Q^{+}(z^*,\rho_1))} &\leq N \||u|^2\|_{L_{3/2}(Q^{+}(z^*,\rho_1))} +N \|p-[p]_{x^*,\rho}\|_{L_{3/2}(Q^{+}(z^*,\rho_1))}\\&\quad
+N\rho_1\|f\|_{L_{3/2}(Q^{+}(z^*,\rho_1))},
\end{align*}
which together with \eqref{ppp}, \eqref{pppp}, and the condition $f\in L_{6}$ yields
\begin{equation}
                        \label{eq10.42}
\int_{Q^{+}(z^*,\rho_1)}|\nabla w|^{3/2}\,dz\le N\rho^{3+5(2-\delta)/6}.
\end{equation}
Since $|\nabla u|\leq |\nabla w|+|\nabla v|$, we combine \eqref{uu} and \eqref{eq10.42} to obtain, for any $r \in (0,\rho/4)$, that
\begin{equation*}
\int_{Q^{+}(z^*,r)}|\nabla u|^{3/2} \,dz \leq N \rho^{3+5(2-\delta)/6}+ r^{6}\rho^{-3/2-\frac{3}{4}\delta}.
\end{equation*}
Upon taking $\delta = \frac{1}{20}$ and $r=\rho^{1000/973}/4$ (with $\rho$ small), we deduce
\begin{equation}
                                \label{eq11.50}
\int_{Q^{+}(z^*,r)}|\nabla u|^{3/2}\,dz \leq N r^{q},
\end{equation}
where
\begin{equation*}q=\frac{36001}{8000}>6-\frac 3 2.
\end{equation*}

Since $u\in \cH^1_{3/2}$ is a weak solution to \eqref{eq11.55}, it then follows from Lemma \ref{lem11.31}, \eqref{eq11.50}, \eqref{ppp}, and \eqref{pppp} with $r$ in place of $\rho$ that
\begin{align}
\nonumber &\int_{Q^{+}(z^*,r)}|u-(u)_{z^*,r}|^{3/2}\,dz \\
\nonumber &\, \leq N r^{3/2}\int_{Q^{+}(z^*,r)}\big|\nabla u|^{3/2}+(|u|^2)^{3/2}+|p-[p]_{x^*,r}|^{3/2}+r^{3/2}|f|^{3/2}\big)\,dz \\
                        \label{eq12.30}
&\, \leq N r^{q+3/2}
\end{align}
for any $r\le r_0$ small and $z^*\in Q({\hat z},\rho_0/4)\cap (\partial\Omega\times({\hat t}-\rho_0^2/16,{\hat t}))$. Allowing $N$ to depend on $r_0$, we get \eqref{eq12.30} for any $r\le \rho_0/4$. Now let $\rho_2\in (0,\rho_0/8)$ be a constant to be specified later. For any $\tilde z=(\tilde t,\tilde x)\in Q({\hat z},\rho_2)$, let  $\tilde r=\tilde x_{4}$ and $z^*=(\tilde x_{1},\tilde x_{2},\tilde x_{3},0,\tilde t)$ be the projection of $\tilde z$ on the boundary of the domain. Note that $z^*\in Q({\hat z},\rho_0/4)\cap (\partial\Omega\times({\hat t}-\rho_0^2/16,{\hat t}))$ . We consider two cases.

{\em Case 1:} $\tilde r<r\le \rho_2$. In this case, we have $\omega(\tilde z,r)\subset Q^+(z^*,2r)$. Thus by \eqref{eq12.30}, we have
\begin{equation}
                               \label{eq12.31}
\int_{\omega(\tilde z,r)}|u-(u)_{\tilde z,r}|^{3/2}\,dz \leq N\int_{Q^{+}(z^*,2r)}|u-(u)_{z^*,2r}|^{3/2}\,dz\leq N r^{q+3/2}.
\end{equation}

{\em Case 2:} $r\le \tilde r<\rho_2$. In this case, we apply the corresponding interior estimate obtained in \cite{Dong_13}. Since $Q(\tilde z,\tilde r)\subset Q^+(z^*,2\tilde r)$, by Proposition \ref{prop2}, we have
$$
C^{2/3}(\tilde r,\tilde z)+D(\tilde r,\tilde z)\le NC^{2/3}(2\tilde r,z^*)+ND(2\tilde r,z^*)\le N\tilde r^{\alpha_0}\le N\rho_2^{\alpha_0}.
$$
Choosing $\rho_2$ sufficiently small, by a similar argument for the interior regularity (see \cite{Dong_13}), we get
\begin{equation}
                               \label{eq12.32}
\int_{\omega(\tilde z,r)}|u-(u)_{\tilde z,r}|^{3/2}\,dz = \int_{Q(\tilde z,r)}|u-(u)_{\tilde z,r}|^{3/2}\,dz\leq N r^{q+3/2}.
\end{equation}
It is worth mentioning that in \cite{Dong_13} an additional scale-invariant quantity $F(r,z)$ was introduced, which is a mixed space-time norm of the pressure $p$, in order to get an initial decay estimate. In view of the proof of Proposition \ref{prop2}, we can avoid using this quantity in the proof.

By Campanato's characterization of H\"older continuous functions near a flat boundary (see, for instance, \cite[Lemma 4.11]{Lieb_96}), from \eqref{eq12.31} and \eqref{eq12.32} we see that $u$ is H\"{o}lder continuous in a neighborhood of ${\hat z}$. This completes the proof of Theorem \ref{th2}.

Theorem \ref{mainthm} then follows from Theorem \ref{th2} by applying Proposition \ref{prop1}. Finally, we can prove that Theorem \ref{mainthm} also holds for a $C^2$ domain similarly by following the argument in \cite{Seregin_04, Mikhailov_11}. We give a brief description of the argument. Since the boundary is $C^2$, we can find a diffeomorphism $\phi$ similar to the one in \cite[Section 2]{Seregin_04} to locally transform the original domain to a domain with flat boundary. Meanwhile, the Navier--Stokes equations become perturbed Navier--Stokes equations. It is crucial that the diffeomorphism $\phi$ is chosen with a smallness condition on the local $C^1$ norm of $\phi$ so that Lemmas \ref{lest_1} and \ref{lest_2} hold true for the perturbed Navier--Stokes equation, and consequently for the original Navier--Stokes equations in the $C^2$ domain (see \cite[Proposition 3.2]{Seregin_04}). Then we can similarly prove Proposition \ref{D_est_lemma}, Lemma \ref{eeeeee}, and Proposition \ref{prop2}, and get the $\epsilon$-regularity criteria.

Theorem \ref{th3} is deduced from Theorem \ref{mainthm} by using the standard argument in the geometric measure theory, which is explained for example in \cite{CKN_82}.

\section{6D stationary case}\label{6d}

For the 6D stationary case, the proof generally follows a similar argument for the 4D time-dependent case, and in fact slightly simpler. We give the proof briefly.
\subsection{Notation and Settings}
                                    \label{6ds1}
In this section, we introduce the notation which are used throughout Section \ref{6d}. Let $\Omega$ be a domain in $\bR^6$. Denote $L_{p}(\Omega;\mathbb{R}^d)$ and $W_p^k(\Omega;\mathbb{R}^d)$ to be the usual Lebesgue and Sobolev spaces of functions from $\Omega$  to $\mathbb{R}^d$.

We shall use the following notation:
\begin{align*}
&B({\hat x},r)=\{x\in\mathbb{R}^6\,|\,|x-{\hat x}|<r\},\ \ B(r)=B(0,r),\ \ B=B(1);\\
&B^{+}({\hat x},r)=\{ x \in B({\hat x},r)\, |\,x=(x',x_6),\,x_6>\hat x_{6} \},\, \,\\&B^{+}(r)=B^{+}(0,r),\ \ B^{+}=B^{+}(1),\ \ \Omega({\hat x},r)=B({\hat x},r)\cap \Omega;\\
& S^+({\hat x},r)=\{x\in \bR^6\,|\,|x-{\hat x}|=r,\,x=(x',x_6),\,x_6>\hat x_{6} \}.
\end{align*}
We also denote mean values of summable functions as follows:
\begin{align*}
[u]_{{\hat x},r}
=\dfrac{1}{|\Omega({\hat x},r)|}\int_{\Omega({\hat x},r)}u(x)\,dx.
\end{align*}

We introduce the following quantities:
\begin{align*}
A(r)&=A(r,{\hat x})=\dfrac{1}{r^4}\int_{\Omega({\hat x},r)}|u|^2\,dx,\\
E(r)&=E(r,{\hat x})=\dfrac{1}{r^2}\int_{\Omega({\hat x},r)}|\nabla u|^2\,dx,\\
C(r)&=C(r,{\hat x})=\dfrac{1}{r^3}\int_{\Omega({\hat x},r)}|u|^3\,dx,\\
D(r)&=D(r,{\hat x})=\dfrac{1}{r^3}\int_{B^{+}({\hat x},r)}|p-[p]_{{\hat x},r}|^{3/2}\,dx,\\
G(r)&=G(r,{\hat x})=r^4\Big[\int_{\Omega({\hat x},r)}|f|^6\,dx\Big]^{1/3}.
\end{align*}
Notice that all these quantities are invariant under the natural scaling:
$$
u_\lambda(x,t)=\lambda u(\lambda x), \quad
p_\lambda(x,t)=\lambda^2 p(\lambda x),\quad
f_\lambda(x,t)=\lambda^3 f(\lambda x).
$$
We shall estimate them in Section \ref{s26}. 
\subsection{Proof}\label{s26}
%


Now we prove the main theorems in three steps.
\subsubsection{Step 1.}First, we control the quantities $A$, $C$, and $D$ in a smaller ball by their values in a larger ball under the assumption that $E$ is sufficiently small. 

\begin{lemma}
Let $\gamma \in (0,1)$, $\rho >0$ be constants and ${\hat x} \in \partial \Omega$. Suppose that ${\hat x}\in \partial\Omega$ and $\Omega({\hat x},\rho)=B^{+}({\hat x},\rho)$. Then we have
\begin{equation*}
\label{C_est6}
C(\gamma \rho) \leq N\big[\gamma^{-3}E^{3/2}(\rho)+\gamma^{-6}A^{3/4}(\rho)E^{3/4}(\rho)+\gamma^3 C(\rho)\big],
\end{equation*}
where $N$ is a constant independent of $\gamma$, $\rho$, and ${\hat x}$.
\end{lemma}
\begin{proof}
Similar to Lemma \ref{lem2.5}, the proof follows that of Lemma 3.1 in \cite{Dong_12}.
\end{proof}

\begin{lemma}
                \label{D_est_lemma6}
Let $\gamma \in (0,1/4]$ and $\rho >0$ be constants. Suppose that ${\hat x}\in \partial\Omega$ and $\Omega({\hat x},\rho)=B^{+}({\hat x},\rho)$. Then 
 we have
\begin{equation}
\label{F_est6}
D(\gamma \rho) \leq N\big[\gamma^{-3}E^{3/4}(\rho)C^{1/2}(\rho)+\gamma^{7/2}D(\rho)
+\gamma^{7/2}E^{3/4}(\rho)+\gamma^{-3}G^{3/4}(\rho)\big],
\end{equation}
where $N$ is a constant independent of $\gamma$, $\rho$, and ${\hat x}$.
\end{lemma}

\begin{proof}
As in the proof of Lemma \ref{D_est_lemma}, we may assume that ${\hat x}=0$ and $\rho=1$. We choose and fix a domain $\tilde{B} \subset \mathbb{R}^{6}$ with smooth boundary so that
\begin{equation*}
B^{+}(1/2) \subset \tilde{B} \subset B^{+}.
\end{equation*}

Next, we decompose $p$ in a way similar to what we did in the proof of Lemma \ref{D_est_lemma}. Let $v$ and $p_1$ be a unique solution to the following initial boundary value problem:
$$
\left\{\begin{aligned}
-\Delta v+\nabla p_1 &= \tilde{f}+f, \quad\text{in}\quad \tilde{B},\\
\nabla \cdot v&=0\quad\text{in}\quad \tilde{B},\\
[p_1]_{\tilde{B}}&=0,\\
v&=0 \quad\text{on}\quad \partial \tilde{B},
\end{aligned}
\right.
$$
where $\tilde{f}=-u\cdot\nabla u$.
We set $w=u-v$ and $p_2=p-p_1-[p]_{0,1/2}$. Then $w$ and $p_2$ satisfy
$$
\left\{
\begin{aligned}
-\Delta w +\nabla p_2 &=0\quad\text{in}\quad \tilde{B},\\
\nabla \cdot w&=0\quad\text{in}\quad \tilde{B},\\
w&=0 \quad\text{on}\quad \partial \tilde{B} \cap \partial\Omega.
\end{aligned}
\right.
$$
Similar to \eqref{v1p1} and \eqref{p2}, we obtain
\begin{align}
\label{v1p16}
\nonumber &\|v\|_{L_{6/5}(\tilde{B})}+\|\nabla v\|_{L_{6/5}(\tilde{B})}+\|p_1\|_{L_{6/5}(\tilde{B})}+\|\nabla p_1\|_{L_{6/5}(\tilde{B})}\\
\nonumber&\,\leq N\|\tilde{f}\|_{L_{6/5}(\tilde{B})}+N\|f\|_{L_{6/5}(\tilde{B})}
\\&\,\leq N\|\nabla u\|_{L_2(B^{+})}\|u\|_{L_3(B^{+})}
+N\|f\|_{L_{6/5}(\tilde{B})},
\end{align}
and
\begin{align*}
&\|\nabla p_2\|_{L_9(B^{+}(1/4))}\\
&\,\leq N[\|w\|_{L_{6/5}(B^{+}(1/2))}+\|\nabla w\|_{L_{6/5}(B^{+}(1/2))}+\|p_2\|_{L_{6/5}(B^{+}(1/2))}]\\
&\,\leq N\big[\|u\|_{L_{6/5}(B^{+}(1/2))}+\|\nabla u\|_{L_{6/5}(B^{+}(1/2))}\\
\nonumber&\,\quad+\|p-[p]_{0,1/2}\|_{L_{6/5}(B^{+}(1/2))}
+\|v\|_{L_{6/5}(B^{+}(1/2))}\\
\nonumber&\,\quad+\|\nabla v\|_{L_{6/5}(B^{+}(1/2))}+\|p_1\|_{L_{6/5}(B^{+}(1/2))}\big].
\end{align*}
Here we took $s=9$, which can be replaced by any sufficiently large number.
Then using \eqref{v1p16}, we have
\begin{align}
\label{p26}
\nonumber &\|\nabla p_2\|_{L_9(B^{+}(1/4))} \\
\nonumber &\,\leq N\big[\|u\|_{L_{6/5}(B^{+}(1/2))}+
\|\nabla u\|_{L_{6/5}(B^{+}(1/2))}+\|p-[p]_{0,1/2}\|_{L_{6/5}(B^{+}(1/2))}\\
&\,\quad+\|\nabla u\|_{L_2(B^{+})}\|u\|_{L_3(B^{+})}
+\|f\|_{L_{6/5}(\tilde{B})}\big].
\end{align}
Recall that $0 < \gamma \leq 1 /4$. Then by the Sobolev--Poincar\'e inequality, the triangle inequality, \eqref{v1p16}, \eqref{p26}, and H\"{o}lder's inequality, we bound $D(\gamma)$ by
\begin{align*}
\nonumber &\dfrac{N}{\gamma^3}\Big(\int_{B^{+}(\gamma)} |\nabla p_1|^{6/5}\,dx\Big)^{5/4}+\dfrac{N}{\gamma^3}\Big(\int_{B^{+}(\gamma)} |\nabla p_2|^{6/5}\,dx\Big)^{5/4}\\
\nonumber&\,\leq N\gamma^{-3}\big[E^{3/4}(1)C^{1/2}(1)\big] +N\gamma^{-3}G^{3/4}+ \gamma^{7/2}\Big(\int_{B^{+}(\gamma)} |\nabla p_2|^9\,dx\Big)^{1/6}
\\&\,\leq N\big[\gamma^{-3}E^{3/4}(1)C^{1/2}(1)+\gamma^{7/2}(D(1)+A^{3/4}(1)
+E^{3/4}(1))+\gamma^{-3}G^{3/4}(1)\big].
\end{align*}
By the boundary Poincar\'e inequality,
$$
A ^{3/4}(1)\le N E^{3/4}(1).
$$
The lemma is proved.
\end{proof}

We omit the proofs of Lemma \ref{ae6} and Proposition \ref{prop16} below as well as Lemma \ref{eeeeee6} at the beginning of Section \ref{sec3.2.2}. The reader is referred to \cite{Dong_12} or Appendix for  details.
\begin{lemma}
\label{ae6}
Let $\gamma \in (0,1/2]$ and $\rho >0$ be constants. Suppose that ${\hat x}\in \partial\Omega$ and $\Omega({\hat x},\rho)=B^{+}({\hat x},\rho)$. Then we have
\begin{equation*}
A(\gamma \rho)+E(\gamma \rho) \leq N \gamma^{-2}\big[C^{2/3}(\rho)+C(\rho)+C^{1/3}(\rho)D^{2/3}(\rho)+G(\rho)\big].
\end{equation*}
In particular, when $\gamma=1/2$ we have
\begin{equation}
\label{AE_est_26}
A(\rho/2)+E(\rho/2) \leq N \big[C^{2/3}(\rho)+C(\rho)+C^{1/3}(\rho)D^{2/3}(\rho)+G(\rho)\big].
\end{equation}
\end{lemma}

As a conclusion, we obtain
\begin{proposition}
\label{prop16}
For any $\epsilon_0 >0$, there exists $\epsilon_1 >0$ small such that the following is true. For any ${\hat x} \in \partial\Omega $ satisfying $\Omega({\hat x},R)=B^{+}({\hat x},R)$ for some small $R$ and
\begin{equation*}
\limsup_{r \searrow 0} E(r)\leq \epsilon_1,
\end{equation*}
we can find $\rho_0$ sufficiently small such that
\begin{equation*}
A(\rho_0)+E(\rho_0)+C(\rho_0)+D(\rho_0)\leq \epsilon_0.
\end{equation*}

\end{proposition}

\subsubsection{Step 2.}
                        \label{sec3.2.2}
In the second step, first we estimate the values of $A$, $E$, $C$, and $D$ in a smaller ball by the values of themselves in a larger ball.
\begin{lemma}
\label{eeeeee6}
Let $\rho >0$ and $\gamma \in (0,1/8]$ be constants. Suppose that ${\hat x}\in \partial\Omega$ and $\Omega({\hat x},\rho)=B^{+}({\hat x},\rho)$. Then we have
\begin{align}
\label{eree6}
&A(\gamma \rho)+E(\gamma \rho)\nonumber\\
&\,\leq N \gamma^2A(\rho)+N\gamma^{-3}\big(C(\rho)+C^{1/3}(\rho)D^{2/3}(\rho)\big)+N\gamma^{-6}G(\rho),
\end{align}
where $N$ is a constant independent of $\rho$, $\gamma$, and ${\hat x}$.
\end{lemma}

In the next proposition, we study the decay property of $A$, $C$, $E$, and $D$ as the radius $\rho$ goes to zero. It is analogous to Proposition \ref{prop2}.
\begin{proposition}
There exists $\epsilon_0>0$ satisfying the following property. Suppose that for some ${\hat x}\in\partial\Omega $ and $\rho_0 >0$, $\Omega({\hat x},\rho_0)=B^{+}({\hat x},\rho_0)$ and
\begin{equation}
C(\rho_0)+D(\rho_0)+G(\rho_0)\leq \epsilon_0.
\label{cond_cdf6}
\end{equation}
Then we can find $N>0$ and $\alpha_0 \in (0,1)$ such that for any $\rho \in (0,\rho_0/4)$ and $x^* \in B({\hat x},\rho_0/4)\cap \partial\Omega$, the following inequality holds uniformly
\begin{equation}
A(\rho,x^*)+C^{2/3}(\rho,x^*)+E(\rho,x^*)+ D(\rho,x^*)\leq N \rho ^{\alpha_0},
\label{cdf6}
\end{equation}
where $N$ is a positive constant independent of $\rho$ and $x^*$.
\end{proposition}

\begin{proof}
Let $\epsilon'>0$ be a small constant to be specified later. Due to \eqref{AE_est_26} and \eqref{cond_cdf6}, one can find $\epsilon_0=\epsilon_0(\epsilon')>0$ sufficiently small such that,
\begin{equation*}
\epsilon_0\le \epsilon'^4,\quad
A(\rho_0/2)+E(\rho_0/2)\leq \epsilon',\quad D(\rho_0/2) \leq \epsilon'.
\end{equation*}
As in the proof of Proposition \ref{prop2}, we may assume that $\rho_0=\epsilon'$.

By the Sobolev embedding theorem, we get
\begin{equation}
\label{hhh6}
C(\rho)\leq N[A(\rho)+E(\rho)]^{3/2},
\end{equation}
which implies
\begin{equation*}
C(\rho_0/2)\leq N\epsilon'^{\frac{3}{2}}.
\end{equation*}

For any $x^* \in B({\hat x},\rho_0/4)\cap \partial\Omega$,
by using
\begin{equation*}
B^{+}(x^*,\rho_0/4)\subset B^{+}({\hat x},\rho_0/2)\subset \Omega,
\end{equation*}
we get
\begin{equation*}
A(\rho_1,x^*)+E(\rho_1,x^*)+C^{2/3}(\rho_1,x^*)+D(\rho_1,x^*) \leq N\epsilon',
\end{equation*}
where $\rho_1=\rho_0/4$.

Now we shall prove inductively that
\begin{equation}
\label{pop6}
A(\rho_k,x^*)+E(\rho_k,x^*)+C^{2/3}(\rho_k,x^*)+D(\rho_k,x^*) \leq \rho_k^{\frac{1}{10}},
\end{equation} where $\rho_k=\rho_1^{(1+\beta)^k}$ and $\beta=\frac{1}{200}$ for $k=1,2,\cdots$.

It is easy to see that \eqref{pop6} holds for $k=1,2$ by choosing $\epsilon'$ sufficiently small.
Now suppose that \eqref{pop6} holds for some $k\ge 2$. Since $\rho_{k+1}=\rho_{k}^{1+\beta}$, by \eqref{eree6} combining with a similar argument in the proof of Proposition \ref{prop2} (cf. \eqref{ind1} to \eqref{ind2}), we get
\begin{align*}
&A(\rho_{k+1})+E(\rho_{k+1}) \\
&\,\leq N \rho_{k}^{2\beta}A(\rho_k)+N \rho_k^{-3\beta}\big(C(\rho_k)+C^{1/3}(\rho_k)D^{2/3}(\rho_k)\big)
+N\rho_k^{-6\beta}G(\rho_k)
\\&\,\leq N \rho_k^{2\beta+\frac{1}{10}}+N\rho_k^{-3\beta+\frac{7}{60}}
+N\rho_k^{-6\beta+4}\\
&\,\leq \rho_{k+1}^{\frac{1}{10}},
\end{align*}
and
\begin{equation*}
C(\rho_{k+1})\leq \rho_{k+1}^{\frac{3}{20}},
\end{equation*}
provided that $\varepsilon'$ is sufficiently small, but independent of $k$.

To bound $D(\rho_{k+1})$, we shall use the estimates which we got in the $k-1$-th step for the same reason we explained in the proof of Proposition \ref{prop2}. By using \eqref{F_est6}  with $\gamma=\rho_{k-1}^{\beta^2+2\beta}$ and $\rho=\rho_{k-1}$, we get
\begin{align*}
D(\rho_{k+1})&\leq N \rho_{k-1}^{-6\beta-3\beta^2}E^{3/4}(\rho_{k-1})C^{1/2}(\rho_{k-1})+ N\rho_{k-1}^{7\beta+7\beta^2/2}D(\rho_{k-1})\\&\quad+N\rho_{k-1}^{7\beta+7\beta^2/2}E^{3/4}(\rho_{k-1})+N \rho_{k-1}^{-6\beta-3\beta^2}G^{3/4}(\rho_{k-1})\\&\leq \rho_{k+1}^{\frac{1}{10}},
\end{align*}
provided that $\varepsilon'$ is sufficiently small, but independent of $k$.

Now for any $\rho \in (0,\rho_0/4)$, we can find a positive integer $k$ such that $\rho_{k+1}\leq \rho <\rho_k$. Therefore,
\begin{align*}
& A(\rho)+E(\rho)+C^{2/3}(\rho)+D(\rho)\\
&\le \rho_k^4 \rho_{k+1}^{-4}\big(A(\rho_k)+E(\rho_k)+C^{2/3}(\rho_k)+D(\rho_k)\big)\\
&\le 2\rho_k^{\frac 1 {10}-4\beta}\leq 2 \rho^{\frac{1}{1+\beta}(\frac 1 {10}-4\beta)}.
\end{align*}
By choosing $\alpha_0=\frac{1}{1+\beta}(\frac 1 {10}-4\beta)$, the lemma is proved.
\end{proof}

\subsubsection{Step 3}

In the final step, we shall use a bootstrap argument to successively improve the decay estimate \eqref{cdf6} and use elliptic regularity to get the decay exponent required by Morrey's Lemma.

First we prove Theorem \ref{th26d}. With the same bootstrap argument used in Section \ref{boot}, we can prove that, under the condition \eqref{cond_cdf6}, for an increasing sequence of real numbers $\{\alpha_k\}$, $\alpha_k \in (0,2)$, and $\alpha_k\rightarrow 2$ as $k\rightarrow +\infty$, the following estimates
\begin{align}
\label{esti6}
&A(\rho,x^*)+E(\rho,x^*)\leq N \rho^{\alpha_k},\quad C(\rho,x^*)\leq N\rho^{3\alpha_k/2},\\
            \label{esti6_D}
&D(\rho,x^*)\leq N\rho^{6\alpha_k/7}
\end{align}
hold uniformly for all $\rho >0$ sufficiently small and $x^* \in B({\hat x},\rho_0/4)\cap \partial \Omega$.
Indeed, the $k=0$ case for \eqref{esti6} and \eqref{esti6_D} was proved in \eqref{cdf6} with a possibly different exponent $\alpha_0$.
Now suppose that \eqref{esti6} and \eqref{esti6_D} hold with the exponent $\alpha_k$. Let $\tilde{\gamma}=\rho^{\mu}$, $\tilde{\rho}=\rho^{1-\mu}$, and $\rho=\tilde{\gamma}\tilde{\rho}$, where $\mu \in (0,1)$ is a constant to be specified. Then we get the following estimate
\begin{equation*}
A(\rho)+E(\rho) \leq N \rho^{2\mu}\rho^{\alpha_k(1-\mu)}+N \rho^{-3\mu}\rho^{\frac{15}{14}\alpha_k(1-\mu)}+N \rho^{4(1-\mu)}\rho^{-6\mu}
\end{equation*}
by using \eqref{eree6} and \eqref{esti6}. Choose $\mu=\dfrac{\alpha_k}{70+\alpha_k}$. Then \eqref{esti6} is proved for $A(\rho)+E(\rho)$ with the exponent
\begin{align*}
\alpha_{k+1}&:=\min\Big\{2\mu+\alpha_k(1-\mu),
\frac{15}{14}\alpha_k(1-\mu)-3\mu,4(1-\mu)-6\mu\Big\}\\
&=\dfrac{72}{70+\alpha_k}\alpha_k \in (\alpha_k,2).
\end{align*}
The estimate for $C(\rho,x^*)$ in \eqref{esti6} with $\alpha_{k+1}$ in place of $\alpha_k$ follows from \eqref{hhh6}.

Similarly, by using \eqref{F_est6}, \eqref{esti6} with $\alpha_{k+1}$ in place of $\alpha_k$, \eqref{esti6_D}, and choosing $\mu=\dfrac{3\alpha_{k+1}}{26+3\alpha_{k+1}}$, we  have
\begin{align*}
D(\rho,x^*)&\leq N\big[\rho^{-3\mu+\frac{3}{2}\alpha_{k+1}(1-\mu)}+\rho^{7\mu/2+\frac{6}{7}\alpha_{k}(1-\mu)}+\rho^{7\mu/2+\frac{3}{4}\alpha_{k+1}(1-\mu)}
\\&\quad+\rho^{-3\mu+3(1-\mu)}\big]\\&\leq N\rho^{\frac{30\alpha_{k+1}}{26+3\alpha_{k+1}}}\leq N\rho^{6\alpha_{k+1}/7}
\end{align*}
since $\alpha_{k+1}\in(0,2)$.
%
%
%

We have got the following estimates for any sufficiently small $\delta$:
\begin{equation}
\int_{B^{+}(x^*,\rho)}|u(x,t)|^2\,dx \leq N \rho^{6-\delta},
\label{pp6}
\end{equation}
\begin{equation}
\int_{B^{+}(x^*,\rho)}|u|^3+|p-[p]_{x^*,\rho}|^{3/2} \,dx \leq N \rho^{3+\frac{6}{7}(2-\delta)}.
\label{ppp6}
\end{equation}

Finally, we use the elliptic regularity theory to improve the decay estimate for $\nabla u$ and then complete the proof.
we rewrite \eqref{ns6d} (in the weak sense) into
\begin{equation*}
- \Delta u_i=-\partial_j(u_iu_j)-\partial_i p+f_i.
\end{equation*}
Due to \eqref{pp6}, there exists $\rho_1 \in (\rho/2,\rho)$ such that
\begin{equation}
\int_{S^{+}(x^*,\rho_1)}|u|^2\,dx \leq N \rho^{5-\delta}.
\label{oo6}
\end{equation}
Let $v$ be the unique weak solution to the Laplace equation
\begin{equation*}
\Delta v=0 \ \ \text{in} \ \ B^{+}(x^*,\rho_1)
\end{equation*}
with the boundary condition $v=u$ on $\partial B^{+}(x^*,\rho_1)$. Since $v=0$ on the flat boundary part, it follows from the standard estimates for the Laplace equation, H\"{o}lder's inequality, and \eqref{oo6} that
\begin{align}
 &\quad\sup_{B^{+}(x^*,\rho_1/2)}|\nabla v| \leq N \rho_1^{-6}\int_{S^{+}(x^*,\rho_1)}|v|\,dx \leq N \rho^{-1-\delta/2}.
\label{uu6}
\end{align}
Denote $w=u-v.$ Then $w$ satisfies the Poisson equation
\begin{equation*}
-\Delta w_i = -\partial_j(u_iu_j)-\partial_i (p-[p]_{x^*,\rho})+f_i \quad \text{in}\quad B^{+}(x^*,\rho_1)
\end{equation*}
with the zero boundary condition. By the classical $L_p$ estimate for the Poisson equation, we have
\begin{align*}
\nonumber\|\nabla w\|_{L_{3/2}(B^{+}(x^*,\rho_1))} &\leq N \||u|^2\|_{L_{3/2}(B^{+}(x^*,\rho_1))} +N \|p-[p]_{x^*,\rho}\|_{L_{3/2}(B^{+}(x^*,\rho_1))}\\&\,\,
+N\rho_1\|f\|_{L_{3/2}(B^{+}(x^*,\rho_1))},
\end{align*}
which together with \eqref{ppp6} and the condition $f\in L_{6}$ yields
\begin{equation}
                        \label{eq10.426}
\int_{B^{+}(x^*,\rho_1)}|\nabla w|^{3/2}\,dx\le N\rho^{3+6(2-\delta)/7}.
\end{equation}
Since $|\nabla u|\leq |\nabla w|+|\nabla v|$, we combine \eqref{uu6} and \eqref{eq10.426} to obtain, for any $r \in (0,\rho/4)$, that
\begin{equation*}
\int_{B^{+}(x^*,r)}|\nabla u|^{3/2} \,dx \leq N \rho^{3+6(2-\delta)/7}+ r^{6}\rho^{-3/2-\frac{3}{4}\delta}.
\end{equation*}
Upon taking $\delta = \frac{1}{20}$ and $r=\rho^{40/39}/4$ (with $\rho$ small), we deduce
\begin{equation*}
\int_{B^{+}(x^*,r)}|\nabla u|^{3/2}\,dx \leq N r^{q},
\end{equation*}
where
\begin{equation*}q=\frac{14403}{3200}>6-\frac 3 2.
\end{equation*}
As in the proof of Theorem \ref{th2}, this together with the corresponding interior estimate obtained in \cite{Dong_12} yields, by Morrey's Lemma, that $u$ is H\"{o}lder continuous in a neighborhood of ${\hat x}\in\partial\Omega$. This completes the proof of Theorem \ref{th26d}.

Theorem \ref{mainthm6d} then follows from Theorem \ref{th26d} by applying Proposition \ref{prop16}. Finally, we can prove that Theorem \ref{mainthm6d} also holds for a $C^2$ domain similarly by following the argument in \cite{Seregin_04, Mikhailov_11}. Theorem \ref{th36d} is deduced from Theorem \ref{mainthm6d} by using the standard argument in the geometric measure theory, which is explained for example in \cite{CKN_82}.

\section*{Appendix}
\begin{proof}[Proof of Lemma \ref{aelemma}]
As before, we assume $\rho=1$. In the energy inequality \eqref{energy}, we set $t={\hat t}$ and choose a suitable smooth cut-off function $\psi$ such that
\begin{align*}
\psi \equiv 0 \ \ \text{in} \ \ Q_{{\hat t}}\setminus Q({\hat z},1),\quad  0\leq \psi \leq 1 \ \ \text{in} \ \ Q_T,\\
\psi \equiv 1 \ \ \text{in} \ \ Q({\hat z},\gamma),  \quad |\partial_t \psi|+|\nabla \psi| +|\nabla^2 \psi| \leq N\ \ \text{in} \ \ Q_{{\hat t}}.
\end{align*}
By using \eqref{energy} and because $u$ is divergence free, we get
\begin{align*}
A(\gamma)+2E(\gamma) &
\leq \dfrac{N}{\gamma^2}\Big[\int_{Q^{+}({\hat z},1)}|u|^2\,dz +\int_{Q^{+}({\hat z},1)}\big(|u|^2+2|p-[p]_{{\hat x},1}|\big)|u|\,dz\\
&\quad+\int_{Q^{+}({\hat z},1)}|f||u|\,dz\Big].
\end{align*}
Using H\"{o}lder's inequality and Young's inequality, one can obtain
\begin{equation*}
\int_{Q^{+}({\hat z},1)}|u|^2\, dz \leq\Big(\int_{Q^{+}({\hat z},1)}|u|^3\, dz\Big)^{2/3}\Big(\int_{Q^{+}({\hat z},1)}\,dz\Big)^{1/3} \leq NC^{2/3}(1),
\end{equation*}
\begin{align*}
&\int_{Q^{+}({\hat z},1)}|p-[p]_{{\hat x},1}||u| \,dz\\
&\leq \Big(\int_{Q^{+}({\hat z},1)}|p-[p]_{{\hat x},1}|^{3/2}\,dz\Big)^{2/3}
\Big(\int_{Q^{+}({\hat z},1)}|u|^3\,dz\Big)^{1/3}\\
&= D^{2/3}(1)C^{1/3}(1),
\end{align*}
and
\begin{align*}
\int_{Q^{+}({\hat z},1)}|f||u| \,dz &\leq \int_{Q^{+}({\hat z},1)}|f|^2\,dz + \int_{Q^{+}({\hat z},1)}|u|^2\,dz\\
&\leq N\Big[\int_{Q^{+}({\hat z},1)}|f|^6\,dz\Big]^{1/3} +  \int_{Q^{+}({\hat z},1)}|u|^2\,dz.
\end{align*}
Then the conclusion follows immediately.
\end{proof}

\begin{proof}[Proof of Proposition \ref{prop1}]
For any $\rho \in (0,R]$ and $\gamma \in (0,1/4)$, by using \eqref{AE_est_2} and Young's inequality,
\begin{equation*}
A(\gamma\rho)+E(\gamma \rho) \leq N\big[C^{2/3}(2\gamma\rho)+C(2\gamma\rho)+D(2\gamma\rho)+G(2\gamma\rho)\big].
\end{equation*}
Combining with \eqref{C_est} and \eqref{D_est} and using Young's inequality again, we have
\begin{align}
\nonumber&A(\gamma \rho)+E(\gamma \rho)+C(\gamma \rho)+D(\gamma \rho)\\
\nonumber&\leq N\big[\gamma^{2/3}C^{2/3}(\rho)+\gamma^{9/4}D(\rho)+\gamma C(\rho)+\gamma A(\rho)\big]\\
\nonumber&\quad +N\gamma^{-100}\big(E(\rho)+E^3(\rho)+G(\rho)\big)+N \gamma^{2/3}\\
\nonumber&\leq N\gamma^{2/3}\big[A(\rho)+E(\rho)+C(\rho)+D(\rho)\big]+N \gamma^{2/3}\\
&\quad +N\gamma^{-100}\big(E(\rho)+E^3(\rho)+G(\rho)\big).
\label{erttt}
\end{align}
It is easy to see that for any $\epsilon_0 > 0$, there are sufficiently small real numbers $\gamma \leq 1/(3N)^{3/2}$ and $\epsilon_1$ such that if \eqref{e_cond} holds then for all small $\rho$ we have
\begin{equation*}
N\gamma^{2/3}+N\gamma^{-100}\big(E(\rho)+E^3(\rho)+G(\rho)\big)<\epsilon_0/2.
\end{equation*}
By using \eqref{erttt} with a standard iteration argument, we obtain
\begin{equation*}
A(\rho_0)+E(\rho_0)+C(\rho_0)+D(\rho_0) \leq \epsilon_0
\end{equation*}
for some $\rho_0 >0$ sufficiently small. 
\end{proof}
\begin{proof}[Proof of Lemma \ref{eeeeee}] As before, we assume $\rho=1$. Define the backward heat kernel as
\begin{equation*} \Gamma(x,t)=\dfrac{1}{4\pi^2(\gamma^2+{\hat t}-t)^2}e^{-\frac{|x-{\hat x}|^2}{2(\gamma^2+{\hat t}-t)}}. \end{equation*}
In the energy inequality \eqref{energy} we put $t={\hat t}$ and choose $\psi=\Gamma\phi$, where $\phi\in C_0^\infty( B({\hat x},1)\times ({\hat t}-1,{\hat t}+1))$ is a suitable smooth cut-off functions satisfying
\begin{align}
\nonumber 0\leq \phi\leq 1 \ \ \text{in} \ \ \mathbb{R}^4\times \bR, \ \ \phi\equiv1 \ \ \text{in} \ \ Q({\hat z},1/2),\\
|\nabla\phi| \leq N, \ \ |\nabla^2 \phi| \leq N \ \ |\partial_t\phi|\leq N\quad \text{in} \ \ \mathbb{R}^4\times\bR. \label{qq}
\end{align}
By using the equality
\begin{equation*}
\Delta \Gamma + \Gamma_t =0,
\end{equation*}
we have
\begin{align} \nonumber&\int_{B^{+}({\hat x},1)}|u(x,t)|^2\Gamma(t,x)\phi(x,t) \,dx +2 \int_{Q^{+}({\hat z},1)}|\nabla u|^2 \Gamma \phi \,dz\\ \nonumber&\leq \int_{Q^{+}({\hat z},1)}\big\{|u|^2(\Gamma \phi_t+\Gamma \Delta \phi+2\nabla\phi\nabla\Gamma)\\ &\quad +(|u|^2+2|p-[p]_{{\hat x},1}|)u\cdot (\Gamma\nabla \phi+\phi\nabla \Gamma)+2|f||u||\Gamma\phi|\big\}\,dz. \label{ggg}
\end{align}
With straightforward computations, it is easy to see the following three properties:\\
(i) For some constant $c>0$, on $\bar{Q^{+}}({\hat z},\gamma)$ it holds that \begin{equation*}
\Gamma \phi = \Gamma \geq c \gamma^{-4}.
\end{equation*}
(ii) For any $z \in Q^{+}({\hat z},1)$, we have \begin{equation*} |\Gamma(z)\phi(z)| \leq N \gamma^{-4},\ \ |\phi(z)\nabla\Gamma(z)|+|\nabla\phi(z)\Gamma(z)| \leq N \gamma^{-5}. \end{equation*}
(iii) For any $z \in Q^{+}({\hat z},1)\setminus Q^{+}({\hat z},\gamma)$, we have \begin{equation*} |\Gamma(z)\phi_t(z)|+|\Gamma(z)\Delta\phi(z)|+|\nabla\phi\nabla\Gamma| \leq N.
\end{equation*}
Then these properties together with \eqref{qq} and \eqref{ggg} yield
\begin{equation*}
A(\gamma)+E(\gamma) \leq N\Big[\gamma^2A(1)+\gamma^{-3}\big(C(1)+C^{1/3}(1)D^{2/3}(1)\big)+\gamma^{-6}G(1)\Big]. \end{equation*}
Thus, the lemma is proved.
\end{proof}

\section*{Acknowledgements.} H. Dong was partially supported by the NSF under agreement DMS-1056737. X. Gu was sponsored by the China Scholarship Council for one year study at Brown University and was partially supported by the NSFC (grant No. 11171072) and the the  Innovation  Program  of  Shanghai  Municipal  Education  Commission  (grant
No. 12ZZ012). The authors would like to thank the referee for his/her very helpful comments.

\end{document}